\documentclass[12pt]{amsart}
\usepackage{amsmath}
\usepackage{amssymb}
\usepackage{bbm}
\newtheorem{theorem}{Theorem}
\newtheorem{ansatz}{Ansatz}
\newtheorem{corollary}[theorem]{Corollary}
\newtheorem{lemma}[theorem]{Lemma}
\usepackage{fullpage}

\newcommand{\inv}{^{-1}}

\usepackage[latin1]{inputenc}
\usepackage{subfigure}

\usepackage{graphicx}

\author{Eric Nordenstam} 
\address{Fakult\"at f\"ur Mathematik, Universit\"at Wien, Garnisongasse 3 / 14, A-1090 Vienna, Austria}
\email{eric.nordenstam@univie.ac.at}
\thanks{Eric Nordenstam was supported by the Austrian Science Foundation FWF, grant Z130-N13.}

\author{Benjamin Young} 
\address{KTH Matematik, 100 44 Stockholm, Sweden.}
\email{benyoung@kth.se}
\thanks{Benjamin Young was supported by grant KAW 2010.0063 from the Knut and Alice Wallenberg
Foundation.}

\title[Correlations for the Novak process ]{Correlations for the Novak process }
\begin{document}
\maketitle
\newcommand{\tv}{B} 
\newcommand{\xv}{A} 

\begin{abstract}
We study random lozenge tilings of a certain shape in the plane called
the \emph{Novak half-hexagon}, and compute the correlation functions for this process. 
This model was introduced by Nordenstam and Young (2011) and has many intriguing
similarities with a more well-studied model, domino tilings of the Aztec diamond. 
The most difficult step in the present paper is to compute the inverse of the 
matrix whose $(i,j)$-entry is the binomial coefficient $C(A, B_j-i)$ for indeterminate variables $A$ and $B_1$, \dots, $B_n$. 

\par

Nous \'etudions des pavages al\'eatoires 
d'une region dans le plan 
par des losanges 
qui s'appelle le \emph{demi-hexagone de Novak} et nous calculons 
les corr\'elations  de ce processus. Ce mod\`ele a \'et\'e introduit par Nordenstam
et Young (2011) et a plusieurs similarit\'es des pavages al\'eatoires 
d'un diamant azt\`eque 
par des dominos. 
La partie la plus difficile de cet article est le calcul de
l'inverse d'une matrice ou l'\'element $(i,j)$ est le coefficient binomial $C(B_j-i, A)$ pour 
des variables $A$ et $B_1$, \dots, $B_n$ indetermin\'es. 
\end{abstract}

\section{Introduction}

This paper is a continuation of the work in~\cite{NoYo11}, in which we
initiated a study of the \emph{Novak half-hexagon} of order $n$.  This is a
roughly trapezoid-shaped planar region (see Figure~\ref{fig:half-hexagon}),
which can be tiled with $3n(n+1)/2$ \emph{lozenges} --- rhombi composed of two
equilateral triangles.  The number of these tilings is computed
in~\cite{NoYo11} to be $2^{n(n+1)/2}$, the same as the well-studied Aztec
diamond (see~\cite{eklp}) and possesses a \emph{domino shuffling algorithm}
closely related to that of the Aztec diamond.  We were able to exploit this
similarity to prove an ``arctic parabola''-type theorem for the Novak
half-hexagon: that with probability tending to 1 as $n \rightarrow \infty$, the
tiling is trivial exterior to a parabola tangent to all three sides of the
figure.

The power-of-two tiling count, the existence of a domino shuffle and the simple limiting shape strongly suggest that it will be tractable to carry out the usual ``next step'' in the study of random tilings: namely, computing \emph{correlation functions} for the tiling.  Loosely speaking, the $k$-point correlation function gives the probability that a fixed set of $k$ lozenges will all be present in a lozenge tiling chosen with respect to the uniform measure on the set of all $2^{n(n+1)/2}$ such tilings.  There are a number of ways to compute these probabilities, all of which rely on the fact that the correlation functions are \emph{determinantal}, meaning that they can be computed as the determinant of a $k \times k$ matrix, whose entries are evaluations of a \emph{correlation kernel}.

If these probabilities can be computed exactly, one can attempt to do
asymptotic analysis of the correlation functions, and demonstrate that the
tiling exhibits \emph{universal} behaviour.  Here, \emph{universal} is a
loaded, technical term coming from statistical mechanics and random matrix
theory:  it means that the correlation functions tend to one of a handful of
well-studied and frequently-occurring limit laws which originally come from
random matrix theory.  For instance, at points near the ``arctic parabola'',
the correlations should tend to the \emph{Airy kernel} (see~\cite{Joh05}) and
in the bulk, they should tend to the \emph{Sine kernel}.  Many point processes
exhibit these limit laws and other related ones, including eigenvalue
distributions of random matrices~\cite{For10}, the Schur process~\cite{OkRe03},
the length of the longest row of a random permutation~\cite{Oko00,BaDeJo99},
continuous Gelfand-Tsetlin patterns~\cite{Met11}, domino tilings of the Aztec
Diamond~\cite{Joh05}, lozenge tilings of the regular hexagon~\cite{Joh05b} and
many more.

\subsection{Results}

In this paper, we compute the correlation kernel for a rather general class of
lozenge tiling problems, of which the half-hexagon is one (we cannot say
anything about its asymptotics yet).  The starting point of our method is the
Eynard-Mehta theorem, explained in Section~\ref{sec:Eynard-Mehta Theorem}.
This is a rather general theorem for computing the correlation functions for
processes which can be described as a product of row-to-row transfer matrices,
as ours can.  The Eynard-Mehta theorem gives the correlation kernel in terms of
the inverse of a certain matrix $M$.  For the half-hexagon, $M$ turns out to be
the \emph{Lindstr\"om-Gessel-Viennot} matrix~\cite{Lindstrom, Gessel-Viennot},
\begin{equation} \label{eqn:LGV matrix} M_{HH} = \left[ \binom{n+1}{2j - i}
\right]_{1 \leq i,j \leq n}, \end{equation} which computes the number of
tilings of the order-$n$ half-hexagon.  In fact, our methods required us to
invert a much more general matrix.

\begin{theorem}
\label{thm:main}
If $\xv, \tv_i (1 \leq i \leq n)$ are parameters and
\begin{equation}
\label{eqn:M definition}
M =\left[ \binom{\xv}{\tv_j - i}\right]_{i,j= 1}^{n},
\end{equation}
then
\begin{equation}
\label{eqn:M inverse}
[M\inv]_{i,j} = 
\binom{\xv+n-1}{\tv_i -1} \inv 
\sum_{k=1}^j \binom{\xv+n-1}{k-1} \binom{\xv-1+j-k}{j-k}
(-1)^{k+j} \prod_{l=1,\, l\neq i }^n \frac{k-\tv_l}{\tv_i - \tv_l}.
\end{equation}
\end{theorem}

Then, the Eynard-Mehta theorem yields the following corollary, which will be shown in Section~\ref{sec:Eynard-Mehta Theorem}.
\begin{corollary}
\label{cor:eynard-mehta correlation functions}
The correlation functions for the Novak half-hexagon are determinental, with kernel given by
\begin{multline}
\label{eq:novak-kernel}
K(r, x; s, y ) =
-\phi_{r,s}(x,y) \\
+ \sum_{i,j=1}^n
\frac{\binom{n+1-r}{2i-x} \binom{s}{y-j}}{
 \binom{2n}{2i-1} }
\sum_{k=1}^{j} \binom{2n}{ k-1} \binom{n+j-k}{j-k} 
\frac{(-1)^{k+j+i+n}}{(i-1)! (n-i)!} 
 \prod_{l=1,\, l\neq i }^n (k-2l)
\end{multline}
where  $\phi\equiv 0$ for $r\geq s$ and 
\begin{equation}
\label{eq:phi}
\phi_{r,s}(x,y) = \binom{s-r}{y-x}
\end{equation}
for $r<s$.% and the Pochhammer symbol $(x)_n = x(x-1)\cdots (x-n+1)$. 
\end{corollary}

\begin{figure}
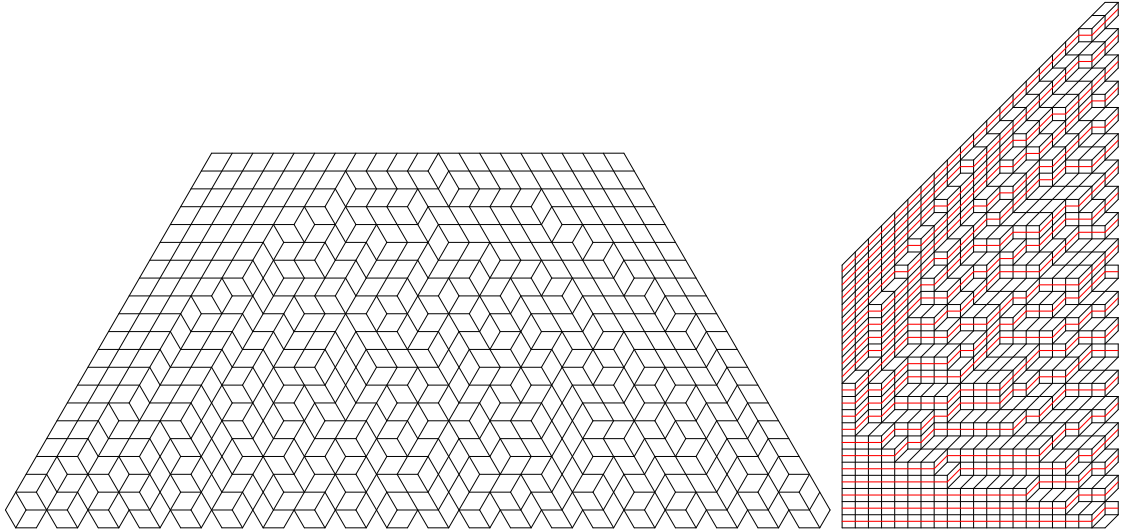

\includegraphics[height=5cm]{novak2.mps} 
\includegraphics[height=7cm]{novak1.mps}
\caption{To the left is a random tiling of an order 20 Novak half-hexagon. To the right is the  same tiling but rotated and skewed. As shown, a tiling corresponds to a set of non-intersecting random walks that at each time-step either stay or jump one unit step. }
\label{fig:half-hexagon}
\end{figure}

\subsection{Inverting a matrix}

Inverting a fixed matrix of numbers is trivial in a computer.  
Symbolically inverting an infinite family of matrices with many parameters is much harder and comprises the bulk of the work in this paper. 

We inverted $M$ with \emph{Cramer's rule}: compute the adjugate matrix $A_{ji}$
(the transposed matrix of cofactors) and divide by the determinant of $M$.
Krattenthaler~\cite{Kra99} gives many methods of evaluating such determinants;
indeed, his Equation (3.12) allows us to compute $\det M$.  Computing the determinant of the matrix of the adjugate matrix, however, is significantly harder, so we first guessed the answer using the computer algebra system
Sage~\cite{Sage}.  The manner in which this guessing was done was itself
nontrivial and may be of interest to others trying to invert matrices; some
details are given in Section~\ref{sec:how we inverted the matrix}.

Once we had conjectured the form of Theorem~\ref{thm:main} and simplified it
considerably, we were able to prove it simply by showing that $M M^{-1}$
is the identity matrix.

\subsection{Related Work}

Metcalfe~\cite{Met11} has developed an alternative approach to problems of this
type, by developing a theory of the asymptotics of a sort of interlacing
particle process.  The theory currently covers a slightly different setting, in
which the positions of the particles is continuous, but Metcalfe is in the process of extending his methods to the discrete setting.

A natural extension of this procedure would be to apply the ideas of
Borodin-Ferrari~\cite{BoFe08} to analyse the dynamics of the domino shuffling
algorithm described in~\cite{NoYo11}.

In~\cite{Joh05b}, there appears a slightly less general kernel, written in
terms of the Hahn polynomials; this is used to prove some theorems on the fluctuations of the frozen boundary of lozenge tilings of a hexagon.  

Acknowledgement:  The authors are extremely grateful to Christian Krattenthaler for many helpful suggestions. Nordenstam wishes to thank Leonid Petrov for some interesting discussions. 

\section{An inverse matrix}
\label{sec:how we inverted the matrix}
Recall that we want to compute the inverse of the matrix $M$ from~\eqref{eqn:M definition} by computing cofactors.  The method of computation is the standard approach of experimental mathematics:  First we guess the answer, making no attempt to be mathematically rigorous.  Then, we prove our guess rigorously, by showing that $MM^{-1}$ is the identity matrix.  As the reader may imagine, the proof alone is not too helpful for guiding people who want to tackle similar matrix inversions in their own work, so we include here an account of how we were able to guess the expression in Theorem~\ref{thm:main}.

Since $M$ is symmetric in its columns, striking out column $k$ simply removes all instances of the variable $\tv_k$ from $M$.  Rename the remaining $\tv$-variables $\bar\tv_1$, \dots, $\bar\tv_{n-1}$. Removing a row, say number $s$,  is more complicated and splits the matrix into two blocks.  To get ready to make our first round of guesses, we take out as many factors as possible so that the matrix elements are now integer polynomials in the variables. The remaining matrix can be written 
\begin{multline}
\det \left[ \binom{\xv}{\bar\tv_j- i - \mathbbm{1}\{i \geq s\}} \right] _{i,j=1}^{n-1}
= \\
\left(\prod_{i=1}^{n-1} \frac{\xv!}{ (\bar\tv_i-1) ! (\xv-\bar\tv_i+n)!} \right)
\det
\overbrace{
\left[
\rule{2mm}{0mm}
\begin{matrix}
\\
(\bar\tv_i -  j + 1) \cdots (\bar\tv_i-1) ( \xv-\bar\tv_i+j+1 ) \cdots (\xv-\bar\tv_i+n)  \\
\\
\hline
\\
(\bar\tv_i -  j ) \cdots (\bar\tv_i-1) ( \xv-\bar\tv_i+j+2 ) \cdots (\xv-\bar\tv_i+n) \\
\\
\end{matrix} 
\rule{2mm}{0mm}\right] 
}^{n-1 \text{ columns}}
\begin{matrix}
\scriptstyle s-1 \\
\rule{0cm}{0.8cm} \\
\scriptstyle n-s 
\end{matrix}.
\end{multline}

 Let $P_{n, s} (\xv, \bar\tv) $ be the value of the second determinant.  Because $P_{n,s}$ is antisymmetric in $\bar\tv_i$, it is divisible by the order $n-1$ Vandermonde determinant; once this is done, the remaining portion is \emph{symmetric}, so we expand it as a (linear!) combination of the elementary symmetric functions $e_{l}$.
We started by computing $P$ for a few different values of the parameters $n$ and $s$. For $s=1$ one quickly conjectures 
\begin{equation}
P_{n, 1} (\xv, \bar\tv) = 
 \Delta(\bar\tv) \left(\prod_{i = 1}^{n-2}(\xv+i)^{n-1-i}\right) 
\left(\prod_{j=1}^{n-1}(\bar\tv_j-1)\right)  
\end{equation} 
where $\Delta$ means taking the Vandermonde determinant in the variables. 
For $s=2$,  \texttt{sage} gave us 
\begin{align*}
P_{3,2} (\xv, \bar\tv) =&
 (\xv+1) \Delta(\bar\tv)  (-2e_2(\bar\tv) + (\xv+4) e_1(\bar\tv) - (3\xv+8))\\
P_{4,2} (\xv, \bar\tv) =&
 (\xv+1)^2(\xv+2) \Delta(\bar\tv) (-3e_3(\bar\tv) + (\xv+6) e_2(\bar\tv) - (3\xv+12)e_1(\bar\tv) + (7\xv+24) )\\
P_{5,2} (\xv, \bar\tv) =&
 (\xv+1)^3(\xv+2)^2(\xv+3) \Delta(\bar\tv) (-4e_4(\bar\tv) +(\xv+8)e_3(\bar\tv) - (3\xv+16) e_2(\bar\tv) \\ &+ (7\xv+32)e_1(\bar\tv) - (15\xv+64) )\\
P_{6,2} (\xv, \bar\tv) =&
 (\xv+1)^4(\xv+2)^3(\xv+3)^2(\xv+4)  \Delta(\bar\tv)(-5e_5(\bar\tv) +(\xv+10)e_4(\bar\tv) \\&- (3\xv+20) e_3(\bar\tv) + (7\xv+40)e_2(\bar\tv) - (15\xv+80)e_1(\bar\tv) +(31\xv+160) )
\end{align*}
Following the immortal advice of David P. Robbins\footnote{``When faced with combinatorial enumeration problems, I have a habit of trying to make the data look similar to Pascal's triangle''.~\cite{Ro91}}, we wrote the coefficients of this four-parameter expression in a tidy fashion, and applied the standard tools in experimental mathematics~\cite{oeis, wikipedia} to all the integer sequences we noticed.
There were many patterns.  For instance, the Stirling numbers of the second kind
$S(n,k)$ appeared in some the coefficients, as did the numbers $n^k$ and
$(n+1)^k-n^k$.  Since the Stirling numbers have the form
\[
S(n,k) = \frac{1}{k!}\sum_{j=0}^k(-1)^{k-j} \binom{k}{j} j^n
\]
and since all of the coefficients we computed seemed to grow exponentially as the index of the elementary symmetric function $l$ decreased, we made the following ansatz: 
\begin{ansatz}
\label{ansatz:exponential growth}
The coefficient of $\xv^ke_{n-1-l}(\bar\tv)$ in $P_{n,s}$ is of the form 
\[
\frac{1}{s!} \sum_{j=0}^s f_{k,l,s,j}(n) j^n,
\] 
where $f_{k,l,s,j}$ is a low-degree polynomial.
\end{ansatz}

We asked \texttt{sage} to find polynomials $f_{k,l,s,j}$ in Ansatz~\ref{ansatz:exponential growth} to fit the data, and to compute more terms.  Computing more terms required heavy optimization of the \texttt{sage} code and, eventually, running the code on a very powerful computer.  After once again writing $f_{k,l,s,j}(n)$ in a tidy table and dividing out some obvious common factors, we noticed a new set of patterns: some of the $f_{k,l,s,j}(n)$ were $i$th derivatives of the \emph{falling factorial functions} $(n-1)(n-2)\cdots(n-k)$.  As such, we made a second ansatz:
\begin{ansatz}
\label{ansatz:falling factorial}
\emph{All} of the $f_{k,l,s,j}(n)$ are linear combinations of falling factorials or their derivatives.
\end{ansatz}
Again, we asked Sage to compute the coefficients of these linear combinations for the data we had.  This time we were able to guess the formula completely.  In the end we conjectured that 
\begin{multline}
\label{eqn:unsimplified_guess}
P_{n,s}(\xv,\bar\tv)  =  \Delta(\bar tv) 
\prod_{r=1}^{n-2} (\xv+r)^{n-1-r}
\times \\
\times
\sum_{l=0}^{n-1} 
\sum_{k=0}^{s-1} 
\sum_{j=1}^{s} 
\sum_{i=0}^j 
(-1)^{n+s+l+j}
\frac{\xv^k j^l e_{n-1-l} (\bar\tv)  }{i!(s-1)!}
s(s-1-j, k-i) \left ( \left(\frac{d}{dn}\right)^i (n-1)\cdots (n-j)\right) 
 \binom{s-1}{j} 
\end{multline}
where $s(n,k)$ are the Stirling numbers of the first kind. 

Obviously, \eqref{eqn:unsimplified_guess} needs to be simplified. By the generating function 
for the Stirling numbers, 
\begin{multline}
\sum_{k=0}^s s(s-1-j, k-i) \xv^k = \xv^i\sum_{k=0}^s s(s-1-j, k-i) \xv^{k-i}\\
 = 
\xv^i [ \xv(\xv-1)\dots (\xv-s+j+2) ] = \xv^i (s-1-j)! \binom{\xv}{s-1-j}.
\end{multline}
By the Binomial Theorem, 
\begin{equation}
\sum_{i=0}^j \frac{\xv^i}{i!} \left( \frac{d}{d\xv} \right) ^i n^\alpha 
=
\sum_{i=0}^j \frac{\alpha(\alpha-1)\cdots (\alpha-i+1)}{i!}  n^{\alpha-i}  = 
\sum_{i=0}^j \binom{\alpha}{i} \xv^i n^{\alpha-i}  = 
 (n+\xv)^\alpha
\end{equation}
by the definition of Binomial coefficients,
\begin{equation}
(n+\xv-1)\cdots (n+\xv-j) = j! \binom{n+\xv-1}{j}
\end{equation}
and lastly, by the generating function for the elementary 
symmetric polynomials, 
\begin{equation}
\sum_{l=0}^{n-1} (-j)^{n-1-l} e_l(\bar\tv) = \prod_{l=1}^{n-1} (\bar\tv_l-j)
\end{equation}

With these simplifications we can write 
\begin{equation}
P_{n,s}(\xv,\bar\tv)  =  \Delta(\bar \tv) 
\prod_{r=1}^{n-2} (\xv+r)^{n-1-r}
\sum_{j=0}^{s-1} 
(-1)^{n+s+j} 
\binom{\xv}{s-1-j} \binom{ n+\xv-1}{j} \prod_{l=1}^{n-1} (\bar\tv_l-j)
\end{equation}

Now to get the inverse matrix we should transpose the cofactor matrix and divide with the determinant of the full matrix. The latter can be found through 
\begin{equation}
\det\left [ \binom{A}{L_i+j} \right ] _{i,j=1}^n = 
\frac{\prod_{1\leq i<j\leq n}  (L_i-L_j) \prod _{i=1}^n (A+i-1)!}{
\prod_{i=1}^n (L_i + n)! \prod_{i=1}^n (A-L_i -1)!}
\end{equation}
which is a special case of \cite[equation (3.12)]{Kra99}. After a bit of simplification, Cramer's rule then leads us to conjecture that~\eqref{eqn:M inverse} is the inverse of $M$. 

\begin{proof}[of Theorem~\ref{thm:main}]
We have now, through these computer experiments, found a formula which we believe expresses $M\inv$. To prove that this guess is correct, we need to show that either $M M\inv = I$ or that $M\inv M = I$ using that formula. One of these (the latter) is easy, the other is hard.   First we write
\begin{multline}
\label{eq:inverseproof1} 
[M M\inv  ] _{\alpha, \gamma}  = \sum_{\beta = 1} ^n
[M]_{\alpha, \beta}[M\inv] _{\beta, \gamma} \\
= \sum_{\beta = 1} ^n \sum_{k=1}^\gamma (-1)^{k+\gamma} \binom{\xv+n-1}{\tv_\beta-1} \inv
\binom{\xv+n-1}{k-1} \binom{\xv-1+\gamma-k}{\gamma-k} \binom{\xv}{\tv_\beta - \alpha} 
\prod_{i=1, \, i\neq \beta}^n \frac{k-\tv_i}{\tv_\beta-\tv_i} .
\end{multline}
Next, we need the following technical lemma, to remove the variables $B_i$ from the equation. 
\begin{lemma}
\label{lem:lagrange}
\begin{equation}
\sum_{\beta = 1} ^n\binom{\xv+n-1}{\tv_\beta-1} \inv\binom{\xv}{\tv_\beta - \alpha} 
\prod_{i=1, \, i\neq \beta}^n \frac{k-\tv_i}{\tv_\beta-\tv_i}  = 
\binom{\xv+n-1}{k-1} \inv \binom{\xv}{k-\gamma}
\end{equation}
\end{lemma}
\begin{proof}
Recall from an undergraduate course how Lagrange interpolation works. 
Let's say you want to fit a polynomial $y=p(x)$ of degree $n-1$ to points
$(x_1, y_1)$, \dots, $(x_n, y_n)$. What you do is you define functions
\begin{equation}
\tau_\beta(x) = \prod _{i=1, \, i\neq \beta}^n \frac{x-x_i}{x_\beta-x_i} 
\end{equation}
and then you compute your polynomial $p$ by
\begin{equation}
p(x) = \sum_{\beta = 1}^n  y_\beta \tau_\beta(x). 
\end{equation}
The sum in the LHS of the Lemma is of exactly this form. Moreover,
\[\binom{\xv+n-1}{t-1} \inv \binom{\xv}{t-\alpha} =
\frac{\xv!}{(\xv+n)!} (t-1)\cdots (t-\alpha + 1) ( \xv-t+n ) \cdots (\xv-t+\alpha+1) \]
is a polynomial of degree $\alpha-1 + n-\alpha = n-1$ in $t$. 
So this sum does Lagrange interpolation of degree $n-1$ to
an expression that is already a polynomial of that degree. 
Replacing the sum with the correct polynomial proves the Lemma. 
\end{proof}

Application of Lemma~\ref{lem:lagrange} reduces~\eqref{eq:inverseproof1} 
to 
\[
[MM\inv ] _{\alpha, \gamma}  =  \sum_{k=1}^\gamma(-1)^{\beta+j} 
 \binom{\xv-1+\beta-k}{\beta-k} \binom{\xv}{k - \gamma}  
\]  
This sum can be computed through  Vandermonde convolution, 
as in~\cite[Equation (5.25)]{GrKnPa94}, showing  that \[
[MM\inv ] _{\alpha, \gamma} = \binom{0}{\alpha-\gamma}  = \delta_{\alpha, \gamma},
\]
which proves that we have indeed found the correct inverse matrix. 
\end{proof}
\section{Eynard-Mehta theorem}
\label{sec:Eynard-Mehta Theorem}
% This figure fits nicely beside the first one, so I put it in the 
% introduction.
%
%\begin{figure}
%\includegraphics[height=8cm]{novak1.mps} 
%\caption{The same realization of the Novak tiling process as in 
%Figure~\ref{fig:half-hexagon}, but rotated and skewed and with the 
%non-intersecting random walks drawn out. }
%\end{figure}

In order to compute correlation functions, we must first describe tilings of the Novak half-hexagon as an ensemble of nonintersecting lattice paths (see Figure~\ref{fig:half-hexagon}).

Consider $n$ walkers on the integer line, started at time 0 at positions $x^{(0)}_1$, $x^{(0)}_2$, 
\dots,  $x^{(0)}_n$.
At time $N$ they end up at positions $x^{(N)}_1$, $x^{(N)}_2$, \dots, 
$x^{(N)}_n$.
 At tick $t$ of the clock they each take a step
according to the transition kernel $\phi_t$. 
In our special case, they either stay where they are or move one step to the right:
\begin{equation}
\label{eq:transition}
\phi_t(x,y) = \delta_{x,y} + \delta_{x+1, y},
\quad \text{$t=0$, \dots, $N-1$.}
\end{equation}
In addition, they are conditioned never to intersect. 
Let the positions of the walkers at time $t$ be  denoted 
 $x^{(t)} = (x^{(t)}_1 <  \dots <  x^{(t)}_n)\in \mathbb{N}^n$ and let 
a full configuration be denoted $x=(x^{(0)}, \dots, x^{(N)})$. 

Then uniform  probability on these configurations can be written
\begin{equation}
p(x) = \frac{1}{Z}\prod_{t=0}^{N-1} \det [\phi_t(x^{(t)}_i, x^{(t+1)}_j ) ]_{i,j=1}^n.
\end{equation}

The normalization constant $Z$ is the total number of configurations. 
For the sake of notation define the convolution product $*$ by 
\[
f * g (x,z) = \sum_{y \in \mathbb{Z}} f(x,y) g(y,z)
\] 
and let 
\[
\phi_{s,t} (x,y)= 
\begin{cases}
(\phi_s*\cdots*\phi_{t-1})(x,y), &  s<t,\\
0, & \text{otherwise.  }
\end{cases}
\]

By the Lindstr\"om-Gessel-Viennot Theorem~\cite{Lindstrom, Gessel-Viennot}, the total number of 
configurations is given by the determinant of the matrix
\begin{equation}
\label{eq:lindstrom-matrix}
M = [ \phi_{0, N} ( x^{(0)}_i,   x^{(N)}_j ) ] _{i,j=1}^N. 
\end{equation}

Correlations can now be computed using the Eynard-Mehta Theorem. 
Readable introductions to it can be found in~\cite[Section 5.9]{For10}
as well as in~\cite{BoRa05}. 
\begin{theorem}[Eynard-Mehta]
Let $m$ be a positive integer and let 
$(t_1, x_1)$, \dots, $(t_m, x_m)$ be a sequence of times and positions. 
The probability that there is a walker at time $t_i$ at position $x_i$ 
for each $i=1$, \dots, $m$ is given by 
\[
\det [ K( t_i, x_i; t_j, x_j) ] _{i,j=1}^m 
\]
where the function $K$, called the kernel of the process,
is given by 
\[
K(r, x; s, y ) = 
- \phi_{r,s}(x,y) + \sum_{i,j=1}^n  \phi_{r, N} (x, x^{(N)}_i  ) [M\inv ] _{i,j} \phi_{0, s} (x^{(0)}_j, y)
\]
\end{theorem}
In our particular case the walkers are going to start densely packed. 
At first we shall leave the end time $N$ and the endpoints unspecified, i.e. 
\begin{align*}
x^{(0)}_i &= i, \\
x^{(N)}_i &= y_i,
\end{align*}
for $i = 1$, \dots, $n$. The particular transition
function~\eqref{eq:transition} gives $\phi_{r,s}$ as defined in~\eqref{eq:phi}.
Inserting that into~\eqref{eq:lindstrom-matrix} gives 
\[
M = \left [ \binom{N}{y_j - i} \right] _{i,j=1}^n,
\]
which is exactly the matrix we inverted in the previous section. 
The kernel can then be written
\begin{multline}
\label{eq:general-kernel}
K(r, x; s, y ) =
-\phi_{r,s}(x,y) \\
+ \sum_{i,j=1}^n
\frac{\binom{N-r}{y_i-x} \binom{s}{y-j}}{ \binom{N+n-1}{y_i-1} }
\sum_{k=1}^{j} \binom{N+n-1}{ k-1} \binom{N-1+j-k}{j-k} (-1)^{k+j} 
\prod_{l=1,\, l\neq i }^n\frac{k-y_l}{y_i-y_l}.
\end{multline}
We state the result in this generality because the 
kernel derived in~\cite{Joh05b} is a special case for suitable
choices of $N$ and $y_i$ in the sense that they are correlation
kernels for the same process. It is not at all clear how to algebraically
relate~\eqref{eq:general-kernel} with the formula in~\cite[Theorem 3.1]{Joh05b}, 
since the latter is a sum involving products of Hahn polynomials.

In our particular case $N = n+1$, 
and the end positions are fixed as $y_i = 2i$ for $i=1$, \dots, $n$.
This specialisation leads to the expression in 
Corollary~\ref{cor:eynard-mehta correlation functions}.

%\bibliographystyle{abbrvnat}
% use the following instead if you encounter problems 
\bibliographystyle{alpha}
\bibliography{references}
\label{sec:biblio}

\end{document}